\documentclass[10pt]{article}

\usepackage{cite}

\usepackage[latin1]{inputenc}  
\usepackage[T1]{fontenc}
\usepackage{ragged2e}

\usepackage{amsmath,amssymb,dcolumn,amsthm}

\usepackage{graphicx,color,transparent}
\graphicspath{{./figs/}}
\usepackage{float}

\usepackage{tikz}
\usetikzlibrary{arrows}

\tikzstyle{block} = [draw, draw=black, line width = 1pt, rectangle,
    rounded corners=4pt,
    minimum height=3em, text width=.7\columnwidth,
    inner sep = 10pt]

\usepackage{algorithm}
\usepackage{algorithmic}

\makeatletter
\renewcommand{\maketag@@@}[1]{\hbox{\m@th\normalsize\normalfont#1}}%
\makeatother

\usepackage{ian_misc,ian_math,ian_opt,ian_abbrvs,ian_color,ian_expl_mpqp,axeIEEEarticle}

\title{\LARGE \bf Reduced Memory Footprint in Multiparametric Quadratic Programming by Exploiting Low Rank Structure}

\author{Isak Nielsen and Daniel Axehill%
  \thanks{I. Nielsen and D. Axehill are with the Division of Automatic
    Control, Linköping University, SE-58183 Linköping, Sweden, \texttt{isak.nielsen@liu.se, daniel.axehill@liu.se}.}}

\begin{document}
\maketitle
\thispagestyle{empty}
\pagestyle{empty}


\begin{abstract}
In multiparametric programming an optimization problem which is dependent on a parameter vector is solved parametrically. In control, multiparametric quadratic programming (\mpQP) problems have become increasingly important since the optimization problem arising in Model Predictive Control (\MPC) can be cast as an \mpQP problem, which is referred to as explicit \MPC. One of the main limitations with \mpQP and explicit \MPC is the amount of memory required to store the parametric solution and the critical regions. In this paper, a method for exploiting low rank structure in the parametric solution of an \mpQP problem in order to reduce the required memory is introduced. The method is based on ideas similar to what is done to exploit low rank modifications in generic \QP solvers, but is here applied to \mpQP problems to save memory. The proposed method has been evaluated experimentally, and for some examples of relevant problems the relative memory reduction is an order of magnitude compared to storing the full parametric solution and critical regions.
\end{abstract}


\section{Introduction}
\label{sec:intro}

In parametric programming the optimization problem is dependent on a parameter which can be thought of as input data to the optimization problem~\cite{bank82:nonlinear_par_opt}. When the optimization problem is dependent on several parameters it is referred to as multiparametric programming, and one class of such problems that has proven to be important is multiparametric quadratic programming (\mpQP) problems. See, \eg,~\cite{bank82:nonlinear_par_opt} for a survey on parametric programming. In control the importance of \mpQP has increased since it was shown in~\cite{Bemporad02:explicit_mpc} that the optimization problem in Model Predictive Control (\MPC) can be cast as an \mpQP problem and solved explicitly. 

\MPC is a control strategy where the control input in each sample is computed as the solution to a constrained finite-time optimal control (\cftoc) problem,~\cite{maciejowski2002predictive}. The \cftoc problem is solved on-line in each sample of the control loop, which requires efficient algorithms for solving the optimization problem. Examples of algorithms where the special structure in \MPC problems is exploited are~\cite{jonson83:thesis,rao98:_applic_inter_point_method_model_predic_contr,nielsen13low_rank_updates}.

Solving the \mpQP problem that corresponds to the \cftoc problem prior to the on-line execution is referred to as explicit \MPC, and the solution is explicitly given as a function of the parameter. For a strictly convex \mpQP problem, the parametric solution is a piecewise affine (\PWA) function of the parameters over polyhedral critical regions,~\cite{Bemporad02:explicit_mpc}. In~\cite{Bemporad02:explicit_mpc}, an algorithm for computing the solution to the \mpQP is presented. The on-line computational effort consists of evaluating the \PWA function for a given parameter~\cite{Bemporad02:explicit_mpc}, which allows for a division free implementation of the control law that can be computed within an a priori known worst case execution time~\cite{kvasnica15:region_free}.

However, there are limitations with \mpQP and explicit \MPC, and much focus in research has been spent to overcome these. The main limitations are the computation of the \PWA function and the critical regions, the computation of a data structure which provides efficient lookup, the memory requirement to store the parametric solution and the critical regions, and the time consumed to determine which critical region the parameter belongs to,~\cite{Bemporad02:explicit_mpc,tondel03:binary_tree,fuchs10:on_the_choice_linear,kvasnica15:region_free}. In~\cite{Bemporad02:explicit_mpc} and~\cite{tondel03:mpQP_and_eMPC} the critical regions and the corresponding optimal active sets are determined by geometric approaches for exploring the parameter space. The algorithm in~\cite{tondel03:mpQP_and_eMPC} exploits the relation between neighboring critical regions and the optimal active sets, and it is reported to avoid unnecessary partitioning. In~\cite{gupta11:novel_mpQP} an approach to solve \mpQP problems by using an implicit enumeration that determines all possible optimal active sets prior to the construction of the critical regions is proposed. The algorithm provides a partition of the full parametric space without unnecessary partitioning. In~\cite{jones06:mpLCP} a method for solving multiparametric linear complementarity problems is presented. This class of problems include \mpQP problems, but also extends to more general problems.

An algorithm which combines explicit \MPC and online \MPC is proposed in~\cite{borrelli10:linear_mpc_laws}. Here, the main algorithm is similar to a standard active set method such as the one presented in, \eg,~\cite{nocedal06:num_opt}, but the search directions are computed offline for all optimal active sets. 
For explicit \MPC it is shown in~\cite{kvasnica12:reducing_memory} that a \PWA function, which is only defined over the regions with non-saturated control inputs, together with a projection onto a non-convex set can be used to reduce the memory required to store the explicit solution. The method of implicitly enumerating all optimal active sets proposed in~\cite{gupta11:novel_mpQP} and the semi-explicit approach in~\cite{borrelli10:linear_mpc_laws} is combined in~\cite{kvasnica15:region_free}, where the authors propose an algorithm that reduces the memory footprint in explicit \MPC.

A commonly used algorithm for improving the on-line process of evaluating the \PWA function is given in~\cite{tondel03:binary_tree}. The authors propose an algorithm based on a binary search tree, which provides evaluation times that are logarithmic in the number of regions. In~\cite{herceg13:eMPC_graph} a graph traversal algorithm is used to evaluate \PWA functions and the graph is constructed while solving the \mpQP problem. In~\cite{fuchs10:on_the_choice_linear} the point location problem is solved by the use of linear decision functions, and significantly better performance in terms of computational time at a small cost of increased memory has been reported.

Furthermore, the Multiparametric Toolbox (\MPT) is an open source \textsc{Matlab}-based toolbox for multiparametric optimization problems,~\cite{MPT3}. 

The main contribution in this paper is the introduction of theory and algorithms for exploiting low rank structure in the parametric solution between neighboring critical regions for an \mpQP problem. The proposed method can significantly reduce the amount of memory required to store the solution and critical regions by exploiting that only low rank modifications of the parametric solution is obtained when making minor changes to the optimal active sets. In methods for solving general \QP problems, exploiting low rank structure has been a crucial approach to improve the performance,~\cite{nocedal06:num_opt}. However, to the authors' knowledge, this has not yet been exploited when storing the solution to \mpQP problems. The method stores the solution in a tree structure and can be incorporated directly in already existing \mpQP solvers, or be applied as a post-processing step to an already existing solution in order to reduce the required memory. Hence, the approach presented here can be interpreted as a data compression algorithm. The problem of solving the \mpQP problem has been considered in previous work by other authors and is outside the scope of this paper. 

In this paper $\posdefmats^n$ ($\possemidefmats^n$) denotes
symmetric positive (semi) definite matrices with $n$ columns, $\intset{i}{j} \triangleq \braces{i,i+1,\hdots,j}$ and symbols in sans-serif font (\eg $\timestack{x}$) denote matrices of stacked components. $|\mathcal{S}|$ denotes the cardinality of the set $\mathcal{S}$.


\section{Multiparametric Quadratic Programming}
\label{sec:mpQP}
In this section the basics of \mpQP are surveyed, and notation that will be used in the following sections is introduced.
Consider an optimization problem in the form
\begin{equation}
\minimize{\frac{1}{2} \U^T \HH \U + \p^T \g \U }{\U}{&\G \U \preceq \bI + \E \p, \quad \p \in \Theta,} \label{eq:mpQP_with_linear_cost_term}
\end{equation}
where $\U \in \realnums{\nz}$ is the optimization variable, $\p \in \realnums{\np}$ is the parameter, the cost function is determined by $\HH \in \posdefmats^{\nz}$ and $\g \in \realnums{\np \times \nz}$, the inequality constraints are given by $\G \in \realnums{\nc \times \nz}$, $\bI \in \realnums{\nc}$ and $\E \in \realnums{\nc \times \np}$, and $\Theta$ is a polyhedral set. The problem~\eqref{eq:mpQP_with_linear_cost_term} is an \mpQP problem with parameter~$\p$, see, \eg,~\cite{bank82:nonlinear_par_opt,Bemporad02:explicit_mpc}. By introducing the change of variables $\z \triangleq \U + \inv{\HH} \g^T \p$, the problem~\eqref{eq:mpQP_with_linear_cost_term} can be transformed into the equivalent \mpQP problem
\begin{equation}
\minimize{\frac{1}{2} \z \HH \z}{\z}{&\G \z  \preceq \bI + \S \p, \quad \p \in \pSet,} \label{eq:mpQP}
\end{equation}
where $\S \triangleq \E + \G \inv{\HH} \g^T \in \realnums{\nc \times \np}$. For any choice of parameter $\p \in \pSet$, the problem~\eqref{eq:mpQP} is a strictly convex \QP problem, and the necessary and sufficient optimality conditions are given by the Karush-Kuhn-Tucker (\KKT) conditions
\begin{subequations}
\label{eq:KKT_conditions}
\begin{align}
\HH \z + \G^T \lam &= 0, \label{eq:KKT_grad} \\
\G \z &\preceq \bI + \S \p, \label{eq:KKT_primal_feasibility} \\
\lam & \succeq 0, \label{eq:KKT_dual_feasibility} \\
\lama{k} \parens{\Ga{k} \z - \ba{k} - \Sa{k} \p} &= 0, \; k \in \intset{1}{\nc}, \label{eq:KKT_complementary_slackness}
\end{align}
\end{subequations}
where $\lam \in \realnums{\nc}$ are the Lagrange multipliers associated with the inequality constraints,~\cite{nocedal06:num_opt}. 

It is shown in, \eg,~\cite{Bemporad02:explicit_mpc} that the solution to the \mpQP problem~\eqref{eq:mpQP} is given by the \PWA function
\begin{equation}
\zopt(\p) = \k{i} + \K{i} \p \; \; \;\textrm{if}\; \; \;\p \in \CR{i}, \; i \in \intset{1}{R}, \label{eq:PWA_solution}
\end{equation}
where $\k{i} \in \realnums{\nz}$, $\K{i} \in \realnums{\nz \times \np}$ define the parametric solution in the polyhedral critical region $\CR{i}$ for $i \in \intset{1}{\R}$.
The \PWA function~\eqref{eq:PWA_solution} and the $\R$ critical regions are computed by solving the \mpQP problem~\eqref{eq:mpQP} parametrically, basically by considering all possible combinations of optimal active constraints, which is explained in, \eg,~\cite{Bemporad02:explicit_mpc,tondel03:mpQP_and_eMPC,gupta11:novel_mpQP}.

\subsection{Compute the solution for an optimal active set}

For any feasible choice of the parameter $\p \in \pSet$, the set of indices of constraints that hold with equality at the optimum is called the \emph{optimal active set}. Assume that no constraints in the optimal active set are weakly active, \ie, no constraints $k$ in the optimal active set have $\lam_k = 0$. How to choose the optimal active sets in the case of weakly active constraints can be seen in, \eg,~\cite{tondel03:mpQP_and_eMPC}. Let the set of all optimal active sets be denoted $\Aset$, and let the elements in $\Aset$ be denoted $\as{i}$ for $i \in \intset{1}{R}$. Each optimal active set $\as{i}$ then corresponds to a critical region $\CR{i}$,~\cite{tondel03:mpQP_and_eMPC,gupta11:novel_mpQP}. Let $\inas{i}$ be the set of inactive constraints, satisfying $\as{i} \cup \inas{i} = \intset{1}{\nc} $ and $\as{i} \cap \inas{i} = \emptyset$.


To clarify the relation between the optimal active set and the corresponding solution and critical region, let $\GA{i}$ and $\GN{i}$ be the matrices consisting of the rows in $\G$ indexed by $\as{i}$ and $\inas{i}$, respectively, and let the same hold for $\SA{i}$, $\SN{i}$, $\bA{i}$ and $\bN{i}$. Furthermore, let $\GA{i}$ have full row rank, \ie, the linear constraint qualification (\licq) holds for $\as{i}$,~\cite{tondel03:mpQP_and_eMPC}. Violation of \licq is referred to as \emph{primal degeneracy}, and if $\GA{i}\z = \bA{i} + \SA{i} \p$ are linearly independent this results in a non full dimensional critical region that in general is a facet between full dimensional regions~\cite{Bemporad02:explicit_mpc,tondel03:mpQP_and_eMPC}, and need hence not be considered here. For any parameter~$\p \in \CR{i}$, where $\CR{i}$ corresponds to $\as{i}$, the solution to the \mpQP problem~\eqref{eq:mpQP} can be computed by
parametrically solving the equality constrained \mpQP problem
\begin{equation}
\minimize{\frac{1}{2} \z \HH \z}{\z}{\GA{i} \z  &= \bA{i} + \SA{i} \p. } \label{eq:mpQP_Ai_eq_constr}
\end{equation}
The parametric solution to~\eqref{eq:mpQP_Ai_eq_constr} is an affine function in the parameter~$\p$.
Now define $\qt{i} \in \realnums{\nc}$ and $\Qt{i} \in \realnums{\nc \times \np}$ as
\begin{equation}
\qt{i} \triangleq \begin{cases}
\qt{i,\inas{i}} = 0, \\
\qt{i,\as{i}} = \q{i},
\end{cases} \quad \Qt{i} \triangleq \begin{cases}
\Qt{i,\inas{i}} = 0, \\
\Qt{i,\as{i}} = \Q{i},
\end{cases} \label{eq:introduction_tilde_notation}
\end{equation}
where $\q{i} \in \realnums{|\as{i}|}$ and $\Q{i}\in \realnums{|\as{i}| \times \np}$ are given by
\begin{equation}
\q{i} + \Q{i} \p \triangleq - \inv{\parens{\GA{i} \inv{\HH} \GA{i}^T}} \parens{\bA{i}  + \SA{i} \p}. \label{eq:qi_Qi_dual_mpQP_Ai_eq_constr}\\
\end{equation}
See, \eg,~\cite{Bemporad02:explicit_mpc} for the details. The ''tilde'' notation, introduced in~\eqref{eq:introduction_tilde_notation}, is used to denote a variable that is related to the $\nc$ constraints in the \mpQP problem~\eqref{eq:mpQP}, but with some components trivially zero. The choice of this notation will later become clear.
The parametric primal and dual solution to~\eqref{eq:mpQP_Ai_eq_constr} for $\p \in \CR{i}$ is then given by
\begin{subequations}
\label{eq:sol_mpQP_Ai_eq_constr}
\begin{align}
&\lamopt{\lam}(\p) = \qt{i} + \Qt{i} \p, \; \;  \label{eq:dual_sol_mpQP_Ai_eq_constr}\\
&\zopt(\p) = \k{i} + \K{i}\p \triangleq -\inv{\HH} \GA{i}^T \parens{\qt{i,\as{i}} + \Qt{i,\as{i}}\p}. \label{eq:primal_sol_mpQP_Ai_eq_constr}
\end{align}
\end{subequations}

\subsection{Compute the critical region for an optimal active set}

The critical region is the set of parameters for which the active set $\as{i}$ is optimal, \ie, all parameters $\p \in \pSet$ such that primal and dual feasibility given by~\eqref{eq:KKT_primal_feasibility} and~\eqref{eq:KKT_dual_feasibility}, respectively, is retained. By inserting the parametric solution~\eqref{eq:sol_mpQP_Ai_eq_constr} in the primal and dual feasibility conditions in the \KKT conditions~\eqref{eq:KKT_conditions}, the critical region $\CR{i}$ is defined by 
\begin{equation}
\CR{i} \!\triangleq \! \big\{\!\p \! \in \! \pSet \, | \, \GN{i}\! \parens{\k{i} \!+\! \K{i} \p }\! \preceq\! \bN{i}\! + \!\SN{i} \p, \; \qt{i,\as{i}}\! + \Qt{i,\as{i}} \p \!\succeq\! 0\big\}, \label{eq:CRi_def}
\end{equation}
which is a polyhedron with $\nc$ hyperplanes,~\cite{Bemporad02:explicit_mpc}. A double index as in, \eg, $\qt{i,\as{i}}$ denotes the components in $\qt{i}$ indexed by $\as{i}$. All hyperplanes in~\eqref{eq:CRi_def} which are redundant can be removed to obtain a description of $\CR{i}$ with the minimal number of describing hyperplanes,~\cite{Bemporad02:explicit_mpc,tondel03:mpQP_and_eMPC}. Let $\ePrimal{i}$ and $\eDual{i}$ be the indices of the describing hyperplanes related to the primal and dual feasibility conditions in~\eqref{eq:CRi_def}, respectively. Then the critical region $\CR{i}$ in~\eqref{eq:CRi_def} is equivalent to
\begin{equation}
\CR{i}\!\! =\! \!\big\{\! \p\! \in\! \pSet \, |\, \Ga{\ePrimal{i}}\!  \parens{\k{i}\! +\! \K{i} \p }\! \preceq\! \ba{\ePrimal{i}} +\! \Sa{\ePrimal{i}} \p, \, \qt{i,\eDual{i}} + \Qt{i,\eDual{i}} \p \!\succeq \!0\! \big\}. \label{eq:CRi_def_minimal}
\end{equation}
Storing the parametric solution and the minimal description of the critical regions requires $\mem{F}$ real numbers, where
\begin{equation}
\mem{F} \! \triangleq\! R\nz\parens{\np\!+\!1} + \memCR{F}, \; \, \memCR{F}\! \triangleq\! \sum_{i=1}^R \!\parens{\ePrimal{i} \!+ \!\eDual{i}} \parens{\np \! + \!1 }. \label{eq:mem_requirement_full}
\end{equation}


\section{Low Rank Changes of Parametric Solution}
\label{sec:low_rank_expl}

In this section it will be shown how the parametric solution in a neighboring region can efficiently be described by small structured modifications of the solution in the first region. Stepping over a facet between two neighboring regions corresponds to adding or removing constraints to the optimal active set,~\cite{tondel03:mpQP_and_eMPC}. Hence, an equivalent interpretation is that the parametric solution for one set of optimal active constraints can be used to describe the solution for an optimal active set where constraints have been added to, or removed from, the first one. For notational convenience the case when only one constraint is added or removed is presented in this paper. The case for $k$ constraints can be shown analogously. 


\subsection{Add one constraint to the optimal active set}
\label{subsec:add_one_constr}
Let the solution corresponding to an optimal active set $\as{i} \in \Aset$ be given by~\eqref{eq:sol_mpQP_Ai_eq_constr}. Consider the case when a constraint $\j \in \inas{i} $ is added, \ie, the active set $\as{j} = \as{i} \cup \j$ is also optimal, and $\as{j}$ hence corresponds to the critical region~$\CR{j}$. 
\begin{theorem}
\label{thm:sol_add_constr_low_rank}
Let the parametric solution of~\eqref{eq:mpQP} for the optimal active set $\as{i}$ be given by~\eqref{eq:sol_mpQP_Ai_eq_constr}. Then, the solution for $\as{j} = \as{i} \cup \j$ with $\j \in \inas{i}$ is given by
\begin{subequations}
\begin{align}
\lamopt{\lam}(\p) &= \qt{j} + \Qt{j}\p \triangleq \qt{i} + \Qt{i} \p - \dt{j} \parens{\cl_j + \vl_j^T \p}, \label{eq:dual_sol_mpQP_Aj_lowrank}\\
\zopt(\p) &= \k{j} + \K{j} \p \triangleq \k{i} + \K{i} \p + \f_j \parens{\cl_j + \vl_j^T \p}, \label{eq:primal_sol_Aj_lowrank}
\end{align}
\end{subequations}
where $\cl_j \in \realnums{}$, $\vl_j \in \realnums{\np}$, $\f_j \in \realnums{\nz}$ and $\dt{j} \in \realnums{\nc}$ with 
$\dt{j,\inas{j}} = 0$, $\dt{j,\j} = -1$ and $\dt{j,\as{i}}$ possibly non-zero. 
\end{theorem}
\begin{proof}
For $\as{j}$, the solution to the \mpQP problem~\eqref{eq:mpQP} is given by~\eqref{eq:sol_mpQP_Ai_eq_constr}, but with index ''$j$'' instead of ''$i$''. Without loss of generality, let $\Ga{\j}$ be the last row in $\GA{j}$. To compute~\eqref{eq:qi_Qi_dual_mpQP_Ai_eq_constr}, the matrix inversion lemma~\eqref{eq:app:mtrx_inversion_lemma} in Appendix~\ref{app:lin_alg} is applied to 
\begin{equation}
\begin{split}
\begin{bmatrix}
\!\W \! \! & \! \! \w \! \\ \! \w^T \! \! & \! \! \wo \!
\end{bmatrix} \!\triangleq \! \GA{j} \inv{\HH} \GA{j}^T \! = \!\begin{bmatrix}
\GA{i} \inv{\HH} \GA{i}^T & \GA{i} \inv{\HH} \Ga{\j}^T \\ \Ga{\j} \inv{\HH} \GA{i}^T & \Ga{\j} \inv{\HH} \Ga{\j}^T 
\end{bmatrix}.
\end{split} \label{eq:def_W_w_w0}
\end{equation}
The dual parametric solution~\eqref{eq:dual_sol_mpQP_Ai_eq_constr} for $\as{j}$ is then given by 
\begin{align}
&\lamopt{\lamA{j}}(\p) = -\begin{bmatrix}
\inv{\W} \parens{\bA{i} + \SA{i} \p}  \\ 0
\end{bmatrix} - \notag \\
&\!\begin{bmatrix}
 \inv{\W} \w \parens{\inv{\C}\w^T \inv{\W}\parens{\bA{i} \!+ \!\SA{i} \p}\! -\! \inv{\C}\! \parens{\ba{\j} \!+\! \Sa{\j} \p}} \\
-\inv{\C} \w^T\inv{\W} \parens{\bA{i} + \SA{i}\p} + \inv{\C} \parens{\ba{\j} + \Sa{\j} \p}
\end{bmatrix} \! \! \! \label{eq:dual_sol_Aj_tmp}
\end{align}
where $\C$ is defined as in Appendix~\ref{app:lin_alg}.
From~\eqref{eq:qi_Qi_dual_mpQP_Ai_eq_constr} and the definition of $\W$ in~\eqref{eq:def_W_w_w0} it is clear that $-\inv{\W} \parens{\bA{i} + \SA{i} \p} = \q{i} + \Q{i}\p$. Furthermore, by defining $\cl_j$ and $\vl_j$ as
\begin{subequations}
\begin{align}
\cl_j &\triangleq \inv{\C} \parens{\w^T \inv{\W} \bA{i} - \ba{\j}} \in \realnums{}, \\
\vl_j &\triangleq \inv{\C} \parens{  \SA{i}^T \inv{\W} \w- \Sa{\j}^T} \in \realnums{\np},
\end{align}
\end{subequations}
the dual solution can compactly be written as~\eqref{eq:dual_sol_mpQP_Aj_lowrank},
where $\dt{j,\as{i}} \triangleq \inv{\W}\w \in \realnums{|\as{i}|} $, $\dt{j,\j} \triangleq -1$ and $\dt{j,\inas{j}} \triangleq 0$.

From~\eqref{eq:primal_sol_mpQP_Ai_eq_constr} it can be seen that the primal solution for the optimal active set $\as{j}$ is $\zopt(\p) = - \inv{\HH} \GA{j}^T \lamopt{\lamA{j}}(\p)$, giving
\begin{align}
&\!\!\!\!\zopt(\p) = -\inv{\HH} \GA{j}^T \parens{ \qt{i,\as{j}} + \Qt{i,\as{j}} \p - \dt{j,\as{j}}\parens{\cl_j + \vl_j^T \p} }\! = \notag\\
&\!\!\!\!\!-\!\!\inv{\HH} \GA{i}^T \!\parens{\qt{i,\as{i}} \!+\! \Qt{i,\as{i}} \p}\! +\! \inv{\HH} \GA{j}^T \!\dt{j,\as{j}}\! \parens{\cl_j\! +\! \vl_j^T \p}\!\!. 
 \label{eq:primal_sol_Aj_tmp}
\end{align}
In the second equality $\qt{i,\j} = 0$ and $\Qt{i,\j} = 0$ by definition are used. Using $ -\inv{\HH} \GA{i}^T \Big(\qt{i,\as{i}} + \Qt{i,\as{i}} \p\Big) = \k{i} + \K{i} \p$ from~\eqref{eq:primal_sol_mpQP_Ai_eq_constr} and defining $\f_j \triangleq \inv{\HH} \GA{j}^T \dt{j,\as{j}} \in \realnums{\nz}$, the primal solution~\eqref{eq:primal_sol_Aj_lowrank} is obtained from~\eqref{eq:primal_sol_Aj_tmp}.
\end{proof}
\begin{remark}
\label{rem:thm1_Ai_to_Aj_low_rank}
From Theorem~\ref{thm:sol_add_constr_low_rank} it can be seen that the parametric solution in $\CR{j}$ can be computed as a rank one modification of the parametric solution in $\CR{i}$.
\end{remark}


\begin{corollary}
\label{cor:def_CR_j_add_constr_low_rank}
Let the parametric solution to~\eqref{eq:mpQP} with $\as{j} = \as{i} \cup \j$ be given by Theorem~\ref{thm:sol_add_constr_low_rank}. Then the corresponding critical region $\CR{j}$ is given by~\eqref{eq:CRi_def}, but where the primal and dual feasibility conditions~\eqref{eq:KKT_primal_feasibility} and~\eqref{eq:KKT_dual_feasibility} are instead given by
\begin{subequations}
\label{eq:feas_Aj_lowrank}
\begin{align}
\GN{j} \parens{\k{i} + \K{i} \p} + \ft{j,\inas{j}} \parens{\cl_j + \vl_j^T \p} &\preceq \bN{j} + \SN{j} \p, \label{eq:primal_feas_Aj_lowrank} \\
\qt{i,\as{j}} + \Qt{i,\as{j}} \p - \dt{j,\as{j}} \parens{\cl_j + \vl_j^T \p} &\succeq 0, \label{eq:dual_feas_Aj_lowrank}
\end{align}
\end{subequations}
where $\ft{j} \triangleq \G \f_j \in \realnums{\nc}$ with $\ft{j,\as{i}} = 0$.
\end{corollary}
\begin{proof}
The critical region $\CR{j}$ is given by~\eqref{eq:CRi_def} but with index ''$j$'' instead of ''$i$''. Hence, the dual feasibility conditions~\eqref{eq:dual_feas_Aj_lowrank} follow directly from the dual solution~\eqref{eq:dual_sol_mpQP_Aj_lowrank}. 

By inserting the primal solution~\eqref{eq:primal_sol_Aj_lowrank} into the inequality constraints of~\eqref{eq:mpQP} and defining $\ft{j} \triangleq \G \f_j$ gives
\begin{equation}
\begin{split}
&\G \parens{\k{i} + \K{i} \p} + 
\G \f_j \parens{\cl_j + \vl_j^T \p} \preceq \bI + \S \p \iff \\ &\GN{j} \parens{\k{i} + \K{i} \p} + \ft{j,\inas{j}} \parens{\cl_j + \vl_j^T \p} \preceq \bN{j} + \SN{j} \p. 
\end{split}
\end{equation}
Here it is used that $\GA{j} \parens{\k{i} + \K{i} \p} + \GA{j} \f_j \parens{\cl_j + \vl_j^T \p} = \GA{j}\parens{\k{j} + \K{j} \p} =  \bA{j} + \SA{j} \p$ by the definition of $\as{j}$. $\ft{j,\as{i}} = 0$ follows from the definition of $\f_j$, $\dt{j}$, $\W$ and $\w$.
\end{proof}
\begin{remark}
\label{rem:cor1_CRi_to_CRj_low_rank}
Note that $\GN{j} \parens{\k{i} + \K{i} \p} \preceq \bN{j} + \SN{j} \p$ are a subset of the primal feasibility conditions in $\CR{i}$, $\qt{i,\j} = 0$ and $\Qt{i,\j} = 0$. Hence, the new information in the description of $\CR{j}$ is contained in $\ft{j,\inas{j}} \in \realnums{|\inas{j}|}$ and $\dt{j,\as{j}} \in \realnums{|\as{j}|}$. Note that $\cl_j$ and $\vl_j$ are already computed for $\lamopt{\lam}(\p)$ and $\zopt(\p)$.
\end{remark}

\subsection{Remove one constraint from the optimal active set}
\label{subsec:remove_one_constr}
When a constraint is removed from the optimal active set, the parametric solution and the description of the critical region change in a similar way as in Section~\ref{subsec:add_one_constr}. 

\begin{theorem}
\label{thm:sol_rem_constr_low_rank}
Let the solution for the optimal active set $\as{j}$ be given by~\eqref{eq:sol_mpQP_Ai_eq_constr} but with index ''$j$'' instead of ''$i$''. Then, the parametric solution for the optimal active set $\as{i} = \as{j} \backslash \j$ with $\j \in \as{j}$ is given by
\begin{subequations}
\begin{align}
\lamopt{\lam}(\p) &= \qt{i} + \Qt{i} \p \triangleq \qt{j} + \Qt{j} \p - \dt{i} \parens{\cl_i + \vl_i^T \p}, \label{eq:dual_sol_rem_Ai_lowrank} \\
\zopt(\p) &= \k{i} + \K{i}\p \triangleq \k{j} + \K{j} \p + \f_i \parens{\cl_i + \vl_i^T \p}, \label{eq:primal_sol_rem_Ai_lowrank}
\end{align}
\end{subequations}
where $\cl_i \triangleq \cl_j$, $\vl_i \triangleq \vl_j$, $\dt{i} \triangleq - \dt{j}$ and $\f_i \triangleq -\f_j$. Here the variables with index ''$j$'' are defined as in Theorem~\ref{thm:sol_add_constr_low_rank}.
\end{theorem}

\begin{proof}
First, note that $\as{i} = \as{j} \backslash \j \iff \as{j} = \as{i} \cup \j$. Hence, Theorem~\ref{thm:sol_add_constr_low_rank} applies for the optimal active set $\as{j}$, and by re-order the terms in~\eqref{eq:dual_sol_mpQP_Aj_lowrank} it can be seen that
\begin{equation}
\qt{i} + \Qt{i}\p = \qt{j} + \Qt{j} \p + \dt{j} \parens{\cl_j + \vl_j^T \p},
\end{equation}
which, by defining the variables $\cl_i \triangleq \cl_j$, $\vl_i \triangleq \vl_j$ and $\dt{i} \triangleq - \dt{j}$ gives~\eqref{eq:dual_sol_rem_Ai_lowrank}. Furthermore, by re-order the terms in~\eqref{eq:primal_sol_Aj_lowrank} it follows that
\begin{equation}
\k{i} + \K{i} \p = \k{j} + \K{j} \p - \f_j \parens{\cl_j + \vl_j^T \p},
\end{equation}
which, using $\cl_i$, $\vl_i$ and defining $\f_i \triangleq - \f_j$, gives~\eqref{eq:primal_sol_rem_Ai_lowrank}.
\end{proof}
Note that, similar to Remark~\ref{rem:thm1_Ai_to_Aj_low_rank}, also here the parametric solution for $\CR{i}$ is a rank one modification of the one in~$\CR{j}$.



\begin{corollary}
\label{cor:def_CR_i_rem_constr_low_rank}
Let the parametric solution to~\eqref{eq:mpQP} with $\as{i} = \as{j} \backslash \j$ be given by Theorem~\ref{thm:sol_rem_constr_low_rank}. Then the corresponding critical region $\CR{i}$ is given by~\eqref{eq:CRi_def}, but where the primal and dual feasibility conditions are instead given by
\begin{subequations}
\label{eq:CR_i_rem_constr}
\begin{align}
\GN{i} \parens{\k{j} + \K{j} \p } + \ft{i,\inas{i}} \parens{\cl_i + \vl_i^T \p } & \preceq \bN{i} + \SN{i} \p,\label{eq:primal_feas_Ai_rem_tmp} \\
\qt{j,\as{i}} + \Qt{j,\as{i}}\p - \dt{i,\as{i}} \parens{\cl_i + \vl_i^T \p} &\succeq 0, \label{eq:dual_feas_Ai_rem_lowrank}
\end{align}
\end{subequations}
where $\ft{i} \triangleq \G \f_i \in \realnums{\nc}$ with $\ft{i,\as{i}} = 0$.
\end{corollary}
\begin{proof}
The dual feasibility conditions~\eqref{eq:dual_feas_Ai_rem_lowrank} follow directly from~\eqref{eq:dual_sol_rem_Ai_lowrank}. Furthermore, inserting the primal parametric solution~\eqref{eq:primal_sol_rem_Ai_lowrank} into the inequality constraints of the \mpQP problem~\eqref{eq:mpQP} gives
\begin{equation}
\GN{i} \parens{\k{j} + \K{j} \p } + \GN{i} \f_i \parens{\cl_i + \vl_i^T \p }  \preceq \bN{i} + \SN{i} \p. \label{eq:thm:proof:primal_feas_Ai_rem_tmp}
\end{equation}
By using the definition $\ft{i} \triangleq \G \f_i $ ,~\eqref{eq:thm:proof:primal_feas_Ai_rem_tmp} gives~\eqref{eq:primal_feas_Ai_rem_tmp}. $\ft{i,\as{i}} = 0$ follows by definition, which concludes the proof.
\end{proof}
\begin{remark}
\label{rem:cor2_CRj_to_CRi_lowrank}
Similar to Remark~\ref{rem:cor1_CRi_to_CRj_low_rank}, the description of $\CR{j}$ is re-used in $\CR{i}$. Furthermore, $\Ga{\j} \parens{\k{j} + \K{j} \p} = \ba{\j} + \Sa{\j} \p$ since $\j \in \as{j}$. Hence, the new information in the description of $\CR{i}$ is contained in $\ft{i,\inas{i}} \in \realnums{|\inas{i}|}$ and $\dt{i,\as{i}} \in \realnums{|\as{i}|}$. 
\end{remark}

\section{Memory Efficient Storage Tree}
\label{sec:low_rank_storage}
In this section it will be shown how the theory presented in Section~\ref{sec:low_rank_expl} can be utilized repeatedly to store the parametric solutions and critical regions in a memory efficient manner. The storage of the parametric solution is arranged into a tree structure, henceforth denoted as the \emph{storage tree}. The tree structure is related to the tree in~\cite{gupta11:novel_mpQP} and the graph in~\cite{herceg13:eMPC_graph}.

The set $\Aset$ consists of all optimal active sets $\as{i}, \, i \in \intset{1}{\R}$ corresponding to the critical regions $\CR{i},\, i \in \intset{1}{\R}$ in the parametric solution of the \mpQP problem~\eqref{eq:mpQP}, and can be arranged in a tree structure by choosing the root node $\rr$ to correspond to $\as{\rr} \in \Aset$ with ${\rr \in \intset{1}{R}}$. 
To simplify the notation, let $\parent{i}$, $\chil{i}$ and $\anc{i}$ denote the parent, the set of children and the ordered set of indices of ancestors of node~$i$ in the tree, respectively. Furthermore, let $\desc{i}$ denote the descendants of node $i$, $\npath{i} = i \cup \parens{\anc{i} \backslash \rr}$ be the ordered set of the node and the ancestor nodes except the root, and let $\tDepth$ be the maximum depth in the tree.
\begin{definition}
\label{def:storage_tree}
The \emph{storage tree} of a set of optimal active sets $\Aset$ is denoted $\sTree{\Aset}{\rr}$, where node $\rr$ is the root node. 
\end{definition}
\begin{assumption}
\label{ass:storage_tree_one_add_rem_constr}
For all nodes $i \in \sTree{\Aset}{\rr} \backslash \rr$ only one constraint is added to, or removed from, the optimal active set of the parent node $\parent{i}$.
\end{assumption}

The nodes in a storage tree from Def.~\ref{def:storage_tree}, for which Ass.~\ref{ass:storage_tree_one_add_rem_constr} holds, correspond to the optimal active sets $\as{i} \in \Aset$, and they are arranged such that either one constraint is added or removed in the optimal active set from a parent to a child. This corresponds to moving across a facet between adjacent critical regions~\cite{tondel03:mpQP_and_eMPC}. As mentioned earlier, the case when $k$ constraints are added or removed can be derived analogously.

\begin{remark}
The results in Section~\ref{sec:low_rank_expl} are not dependent on the tree structure and hold also for more general graph structures than the one chosen here.
\end{remark}

An example with two parameters and a partitioning consisting of $6$ critical regions $\CR{i}$ for $i \in \intset{1}{6}$ is seen in Figure~\ref{fig:cr_2d}, and a corresponding storage tree $\sTree{\Aset}{2}$ is seen in Figure~\ref{fig:storage_tree}. Each critical region $\CR{i}$ corresponds to the optimal active set $\as{i} \in \Aset$, where $\Aset = \{\{\},\{1\},\{2\},\{3\},\{1,2\},\{1,3\} \}$. Hence, $\CR{2}$ is the critical region for the optimal active set $\as{2} = \{1\}$ etc. In Figure~\ref{fig:cr_2d}, the hyperplane between, \eg, $\CR{2}$ and $\CR{6}$ corresponds to constraint $3$. Hence, moving from region $\CR{2}$ to $\CR{6}$ by stepping over the shared facet corresponds to adding constraint~$3$ to the optimal active set in $\CR{2}$, \ie, $\as{6} = \as{2} \cup \{3\}$. In the tree in Figure~\ref{fig:storage_tree} the optimal active set $\as{2} = \{1\}$ is chosen as root node, and for this tree it can be seen that, for example, $\parent{3} = 1$, $\anc{3} = \{1,2\}$, $\npath{3} = \{3,1\}$ and $\chil{2} = \{1,5,6\}$. The maximum depth is $\tDepth = 2$. The transition from $\CR{2}$ to $\CR{6}$ corresponds to moving from node $2$ to node $6$ in the tree by adding constraint $3$.
 \begin{figure}
 \centering
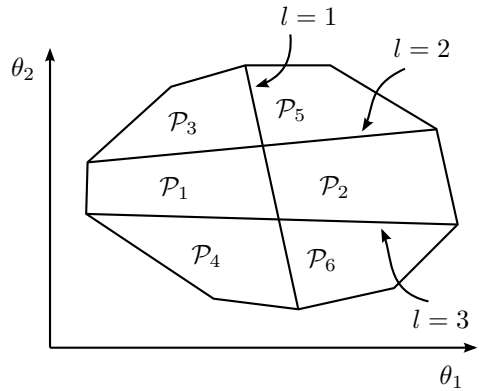
\caption{Example with critical regions $\CR{i}$ for $i \in \intset{1}{6}$ in two dimensions. The three separating hyperplanes correspond to constraints $1$, $2$ and $3$, respectively. When moving between adjacent critical regions, the constraint corresponding to the facet is either added or removed from the optimal active set. An example of a storage tree for this partition is seen in Figure~\ref{fig:storage_tree}.}
\label{fig:cr_2d}
\end{figure}
\begin{figure}
\centering
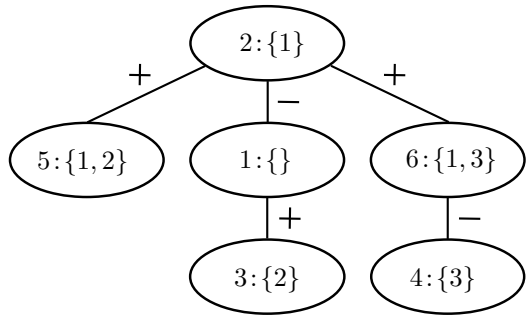
\caption{Example of a storage tree for the partitioning in Figure~\ref{fig:cr_2d}, where each $\CR{i}$ for $i \in \intset{1}{6}$ corresponds to an optimal active set in $\Aset = \{\{\},\{1\},\{2\},\{3\},\{1,2\},\{1,3\} \}$. The number in front of the '':'' in each node corresponds to the index $i$ in $\Aset$. A ''+'' sign at an edge corresponds to adding a constraint to the child, and a ''-'' sign corresponds to removing a constraint. Note that the tree structure is not unique.}
\label{fig:storage_tree}
\end{figure}

Note that the storage tree is not unique. Here, \eg, the tree $\sTree{\Aset}{1}$ could also be used, \ie, having the optimal active set $\{\}$ in the root instead. The choice of tree structure will affect the maximum depth of the tree, and hence also the on-line performance. How to choose the tree to obtain maximum performance is outside the scope of this work.

\begin{remark}
The storage tree could either be constructed after all optimal active sets have been determined, or while building the solution to the \mpQP problem in the solver.
\end{remark}

\subsection{Compute the parametric solution and critical region}
\label{subsec:compute_sol_and_cr_in_storage_tree}

By storing the full description of the critical region $\CR{\rr}$ and the parametric solution $\k{\rr}$ and $\K{\rr}$ in the root node and repeating the theory in Section~\ref{sec:low_rank_expl}, it can be shown that the parametric solution and description of the critical region in each node $i$ can be described by a number of low rank modifications of the root node. From the definition of $\sTree{\Aset}{\rr}$ it can be seen that the modifications which are used to obtain the solution and the critical region in node $i$ are stored in the nodes $j \in \npath{i}$, \ie, the nodes along the path from the root node to node $i$. 

\begin{theorem}
\label{thm:compute_node_low_rank}
Let $\Aset$ be the set of optimal active sets for the \mpQP problem~\eqref{eq:mpQP}, and let $\sTree{\Aset}{\rr}$ be a storage tree for which Ass.~\ref{ass:storage_tree_one_add_rem_constr} holds. 
Then the parametric solution for $\p \in \CR{i}$ for $i \in \sTree{\Aset}{\rr}$ is given by
\begin{equation}
\zopt(\p) = \k{\rr} + \K{\rr}\p + \sum_{j \in \npath{i}} \f_j \parens{\cl_j + \vl_j^T \p}, \label{eq:thm:compute_node_lr_primal_sol}
\end{equation}
where $\cl_j$, $\vl_j$ and $\f_j$ are defined as in Theorem~\ref{thm:sol_add_constr_low_rank} or Theorem~\ref{thm:sol_rem_constr_low_rank} depending on the type of edge between node $i$ and $\parent{i}$.
Furthermore, the critical region $\CR{i}$ is described by the set of parameters $\p \in \pSet$ that satisfy the inequalities
\begin{subequations}
\label{eq:thm:compute_node_lr_feas}
\begin{align}
\btp{\inas{i}} + \Atp{\inas{i}} \p + \sum_{j \in \npath{i}}   \ft{j,\inas{i}} \parens{\cl_j + \vl_j^T \p } &\preceq 0,\label{eq:thm:compute_node_lr_primal_feas} \\
-\qt{\rr,\as{i}} - \Qt{\rr,{\as{i}}} \p  + \sum_{j \in \npath{i}} \dt{j,\as{i}} \parens{\cl_j + \vl_j^T \p } &\preceq 0, \label{eq:thm:compute_node_lr_dual_feas}
\end{align}
\end{subequations}
where $\btp{} \triangleq \G \k{\rr} - \bI$ and $\Atp{} \triangleq \G \K{\rr}- \S$.
\end{theorem}
%
\begin{proof}
Assume that~\eqref{eq:thm:compute_node_lr_primal_sol} holds for some $n \in \sTree{\Aset}{\rr}$ such that $\chil{n} \neq \emptyset$. Take an arbitrary node $i \in \chil{n}$ where $\as{i} = \as{n} \cup \j$ (or $\as{i} = \as{n} \backslash \j$). Then it follows from Theorem~\ref{thm:sol_add_constr_low_rank} (or Theorem~\ref{thm:sol_rem_constr_low_rank}) that the parametric solution is
\begin{align}
&\zopt(\p) = \k{\rr} + \K{\rr}\p + \sum_{j \in \npath{n}} \f_j \parens{\cl_j + \vl_j^T \p} \notag +\\& \f_i \parens{\cl_i + \vl_i^T \p} = 
\k{\rr} + \K{\rr}\p + \sum_{j \in \npath{i}} \f_j \parens{\cl_j + \vl_j^T \p},
\end{align}
where it has been used that $\npath{i} = i \cup \npath{n}$ since $i \in \chil{n}$. Since~\eqref{eq:thm:compute_node_lr_primal_sol} holds for the root $\rr$, and $n$ and $i\in \chil{n}$ were chosen arbitrary, the relation~\eqref{eq:thm:compute_node_lr_primal_sol} follows from induction.

The relations~\eqref{eq:thm:compute_node_lr_feas} can be shown analogously from Corollary~\ref{cor:def_CR_j_add_constr_low_rank} (or Corollary~\ref{cor:def_CR_i_rem_constr_low_rank}) by utilizing that, for the root node $\rr$, the dual feasibility conditions~\eqref{eq:KKT_dual_feasibility} are $-\qt{\rr} - \Qt{\rr} \p  \preceq 0$ and the primal feasibility conditions~\eqref{eq:KKT_primal_feasibility} are described by
\begin{align}
\G \parens{\k{\rr} + \K{\rr} \p} \preceq \bI + \S\p \iff  \btp{} + \Atp{} \p &\preceq 0,
\end{align}
where $\btp{} \triangleq \G \k{\rr} - \bI$ and $\Atp{} \triangleq \G \K{\rr}- \S$.
\end{proof}

\begin{remark}
\label{rem:only_store_full_root}
The full parametric solution and description of the critical region is only stored for the root node $\rr$. For the rest of the nodes, only the low rank modifications are stored. 
\end{remark}

The method to compute the solution as in Theorem~\ref{thm:compute_node_low_rank} is implemented in Alg.~\ref{alg:eval_sol}, where the parameter $\p$ and the node $i$ are inputs, and the optimal solution $\zopt(\p)$ is returned.
\begin{algorithm}[htbp!]
  \caption{On-line evaluation of parametric solution}
  \label{alg:eval_sol}
  \begin{algorithmic}[1]
  \INPUT $\p$ and $i$.
  \STATE Initialize $z = \k{\rr} + \K{\rr}\p.$
    \FOR{$j \in \npath{i}$}
	\STATE $z:=z+\f_j \parens{\cl_j + \vl_j^T \p}$.
    \ENDFOR
    \STATE $\zopt := z$.
  \end{algorithmic}
\end{algorithm}

\begin{corollary}
\label{cor:compute_node_CR_low_rank_eHyp}
To obtain the minimal representation of the critical region $\CR{i}$, only the describing hyperplanes $\ePrimal{i}$ and $\eDual{i}$ need to be used, \ie, $\CR{i}$ can be described by~\eqref{eq:CRi_def_minimal} but where the primal and dual feasibility conditions are given by
\begin{subequations}
\label{eq:cor_compute_node_CR_lr_ineqs}
\begin{align}
\btp{\ePrimal{i}} + \Atp{\ePrimal{i}} \p + \sum_{j \in \npath{i}}  \ft{j,\ePrimal{i}} \parens{\cl_j + \vl_j^T \p } &\preceq 0, \\
-\qt{\rr,\eDual{i}} - \Qt{\rr,\eDual{i}} \p  + \sum_{j \in \npath{i}}   \dt{j,\eDual{i}} \parens{\cl_j + \vl_j^T \p } &\preceq 0.
\end{align}
\end{subequations}
\end{corollary}
\begin{proof}
The corollary follows directly from Theorem~\ref{thm:compute_node_low_rank} since each hyperplane $n$ in~\eqref{eq:thm:compute_node_lr_feas} is described only by the components and rows indexed by $n$ in~\eqref{eq:thm:compute_node_lr_feas}.
\end{proof}

In Alg.~\ref{alg:eval_hyper} the evaluation of a hyperplane using Corollary~\ref{cor:compute_node_CR_low_rank_eHyp} is implemented. The parameter $\p$, the node $i$ and the hyperplane index $n$ are given as inputs to the algorithm, and the value of the hyperplane, denoted $s$, is returned.
\begin{algorithm}[htbp!]
  \caption{On-line evaluation of a hyperplane}
  \label{alg:eval_hyper}
  \begin{algorithmic}[1]
  \INPUT $\p$, $i$ and $n$.
    \IF{$n \in \ePrimal{i}$}
    \STATE Initialize $s := \btp{n} + \Atp{n}\p.$
    \FOR{$j \in \npath{i}$}
	\STATE $s:=s+\ft{j,n} \parens{\cl_j + \vl_j^T \p}$.
	\ENDFOR
	\ELSE 
	\STATE Initialize $s:=-\qt{\rr,n} - \Qt{\rr,n}\p$
	\FOR{$j \in \npath{i}$}
	\STATE $s:=s+\dt{j,n} \parens{\cl_j + \vl_j^T \p}$.
	\ENDFOR
	\ENDIF
  \end{algorithmic}
\end{algorithm}
Note that for a given parameter $\p$ the value of $\cl_j + \vl^T_j \p$ only has to be computed once for each $\CR{j}$ when evaluating hyperplanes, and $\btp{n} + \Atp{n}\p$ and $-\qt{\rr,n} - \Qt{\rr,n}\p$ need only be computed once for each parameter and can be re-used by all nodes.

\begin{remark}
In algs.~\ref{alg:eval_sol} and~\ref{alg:eval_hyper} the order of summation can be chosen to facilitate on-line performance by taking, \eg, memory access into consideration.
\end{remark}

\subsection{Storing the parametric solution and critical region}
\label{subsec:reducing_memory_storage_tree}
From Theorem~\ref{thm:compute_node_low_rank} it follows that computing the parametric solution in node $i$ only requires the storage of $\cl_j \in \realnums{}$, $\vl_j \in \realnums{\np}$ and $\f_j \in \realnums{\nz}$ for each node $j \in \npath{i}$. However, in the root node $\rr$ the full parametric solution defined by $\k{\rr} \in \realnums{\nz}$ and $\K{\rr} \in \realnums{\nz \times \np}$ needs to be stored. 
In Corollary~\ref{cor:compute_node_CR_low_rank_eHyp} it is shown that the description of the critical region in node $i$ only requires the describing hyperplanes. Hence, the full vectors and matrices in~\eqref{eq:cor_compute_node_CR_lr_ineqs} need not be stored, but only the components and rows of $\btp{}$, $\qt{\rr}$, $\Atp{}$, $\Qt{\rr}$, $\ft{j}$ and $\dt{j}$ for $j \in \npath{i}$ that correspond to the describing hyperplanes. 

Furthermore, from Corollary~\ref{cor:compute_node_CR_low_rank_eHyp} it can be seen that the low rank modification in node~$i$ is also used by all descendants $\desc{i} \subset \sTree{\Aset}{\rr}$. Hence, the entries and rows with indices corresponding to defining hyperplanes for $\CR{j}$ with $j \in \desc{i}$ also need to be stored in node~$i$.
Let $\sPrimal{i}$ and $\sDual{i}$ be the indices of the hyperplanes corresponding to primal and dual feasibility conditions in the description of $\CR{i}$ that need to be stored in node~$i$. Then these sets of indices are 
\begin{align}
\sPrimal{i} &\triangleq \bigcup_{j \in \desc{i} \cup i } \ePrimal{j} , \label{eq:def_sPrimal} \quad
\sDual{i} \triangleq \bigcup_{j \in \desc{i} \cup i } \eDual{j}. 
\end{align}
Here, by definition, the trivial zeros in $\ft{i}$ and $\dt{i}$ should not be stored when computing $\sPrimal{i}$ and $\sDual{i}$ in~\eqref{eq:def_sPrimal}.
Hence, $\cl_i \in \realnums{}$, $\vl_i \in \realnums{\np}$, $\f_i \in \realnums{\nz}$, $\ft{i,\sPrimal{i}} \in \realnums{|\sPrimal{i}|}$ and $\dt{i,\sDual{i}} \in \realnums{|\sDual{i}|}$ need to be stored for each node $i\in \sTree{\Aset}{\rr} \backslash \rr $. Storing these for node~$i$ requires $1 + \np + \nz + |\sPrimal{i}| + |\sDual{i}|$ real numbers. For the root node $\rr$, the vectors $\k{\rr} \in \realnums{\nz}$, $\btp{\sPrimal{\rr}} \in \realnums{|\sPrimal{\rr}|}$, $\qt{\rr,\sDual{\rr}} \in \realnums{|\sDual{\rr}|}$ and the matrices $\K{\rr} \in \realnums{\nz \times \np}$, $\Atp{\sPrimal{\rr}} \in \realnums{|\sPrimal{\rr}| \times \np}$ and $\Qt{\rr,\sDual{\rr}} \in \realnums{|\sDual{\rr}| \times \np}$ need to be stored. This requires $m\parens{\np+1} +\parens{|\sPrimal{\rr}|+|\sDual{\rr}| } \parens{\np + 1}$ real numbers to be stored. Hence, the total number $\mem{LR}$ of stored real numbers for $\sTree{\Aset}{\rr}$ is
\begin{align}
\mem{LR} &= \nz \parens{\np +R } + \memCR{LR}, \; \; \, \memCR{LR} \triangleq \parens{ |\sPrimal{\rr}| + |\sDual{\rr}|} \parens{\np +1 } + \notag\\& \parens{R-1}\parens{1 + \np } +\sum_{i \in \sTree{\Aset}{\rr} \backslash \rr} \parens{|\sPrimal{i}| + |\sDual{i}|}. \label{eq:mem_requirement_LR}
\end{align}
Note that the storage of $\cl_i $ and $\vl_i $ is included in $\memCR{LR}$.

\begin{remark}
Note that the set of optimal active sets $\Aset$ can be described by several smaller trees, where the union of the node in the trees corresponds to all optimal active sets $\as{i} \in \Aset$ for $i \in \intset{1}{R}$. In the example presented in Figure~\ref{fig:cr_2d} and~Figure~\ref{fig:storage_tree}, the partitioning can for example be described by two trees; one containing, \eg, the nodes $1,2,3,5$ and one containing the nodes $4,6$. In this case, there are two root nodes (one for each tree) where the full solution and description of the critical region need to be stored, \ie, several trees require more memory to store the parametric solution and critical regions. Using several trees instead of one can, if chosen correctly, reduce the maximum depth of each tree and hence affect the on-line performance. The balancing between the choice of number of trees and the structure of each tree to reduce the memory requirements and the on-line performance to evaluate a solution or a hyperplane is probably problem dependent, and the choice can be made by optimizing the memory reduction and on-line complexity to fit the requirements of a particular hardware set-up and problem. How to do this is not investigated in this work.
\end{remark}

\section{Explicit Model Predictive Control}
\label{sec:explicit_MPC}

In linear~\MPC the input is computed by solving a \cftoc problem. A common formulation of the \cftoc problem is
\begin{equation}
\minimize{\frac{1}{2} \sum_{t=0}^{N-1} \parens{x_t^T \QMPC x_t + u_t^T \RMPC u_t} + \frac{1}{2} x_N^T \PN x_N}{\timestack{x},\timestack{u}}{&x_0 = \xinit \\ &x_{t+1} = \AMPC x_t + \BMPC u_t, \; t \in \intset{1}{N-1} \\ &\HxMPC x_t + \HuMPC u_t + \hMPC \preceq 0, \; t \in \intset{1}{N-1} \\ &\HxMPC x_N + \hMPC \preceq 0,} \label{eq:cftoc}
\end{equation}
where $x_t \in \realnums{\nx}$ is the state vector, $u_t \in \realnums{\nuu}$ is the control input, $\xinit$ is the initial state and $N$ is the prediction horizon. The cost function is given by $\QMPC \in \possemidefmats^{\nx}$, $\RMPC \in \posdefmats^{\nuu}$ and $\PN \in \possemidefmats^{\nx}$. The equality constraints are the dynamics equations of the controlled system, and the inequality constraints are constraints on the states and control inputs.


Similarly to what is shown in~\cite{Bemporad02:explicit_mpc}, the \cftoc problem~\eqref{eq:cftoc} can equivalently be written in the form of an \mpQP problem~\eqref{eq:mpQP_with_linear_cost_term} by defining the matrices
\begin{subequations}
\begin{align}
\HH &\triangleq \timestack{Q_u} + \timestack{B}^T \timestack{Q_x} \timestack{B} \in \posdefmats^{N\nuu}, \; \g \triangleq \timestack{A}^T \timestack{Q_x}\timestack{B}, \\ \G &\triangleq \timestack{H_x} \timestack{B} + \timestack{H_u}, \; \bI \triangleq - \timestack{h}, \; \E \triangleq - \timestack{H_x} \timestack{A},\; \p \triangleq \xinit,
\end{align}
\end{subequations}
where $\timestack{Q_x}$, $\timestack{Q_u}$, $\timestack{A}$, $\timestack{B}$, $\timestack{H_x}$, $\timestack{H_u}$ and $\timestack{h}$ are all defined in~\eqref{eq:app:MPC_mtrx_def} in Appendix~\ref{app:lin_alg}.
%
%
By re-writing the \mpQP problem~\eqref{eq:mpQP_with_linear_cost_term} into~\eqref{eq:mpQP} and solving it parametrically, the optimal solution to the \cftoc problem is given by $\timestack{u}^*(\p) = \zopt(\p) - \inv{H} \g^T \p$ where $\zopt(\p)$ is the \PWA function~\eqref{eq:PWA_solution},~\cite{Bemporad02:explicit_mpc}. 

Since the \cftoc problem is equivalent to an \mpQP problem in the form~\eqref{eq:mpQP}, the storage of the explicit solution to the \cftoc problem can be done using the theory presented in this paper. Only the first control input $u_0^*$ is used as input to the plant in the \MPC control loop~\cite{maciejowski2002predictive}, and hence the full parametric solution does not need to be stored in the case of explicit \MPC. Here only the first $\nuu$ components and rows of $\k{i}$ and $\K{i}$, respectively, are stored. For the traditional non-compressed solution this results in that
\begin{equation}
\memMPC{F} = R\nuu\parens{\np+1} + \sum_{i=1}^R \parens{\ePrimal{i} + \eDual{i}} \parens{\np +1 }, \label{eq:mem_requirement_mpc_full}
\end{equation}
real numbers are stored. This is a slightly modified version of $\mem{F}$ in~\eqref{eq:mem_requirement_full}. Similarly for the compressed solution, only the low rank modifications affecting the first $\nuu$ rows need to be stored in $\sTree{\Aset}{\rr}$. Hence, for the explicit \MPC solution, the number of stored real numbers is
\begin{equation}
\memMPC{LR} = \nuu \parens{\np +R } + \memCR{LR}, \label{eq:mem_requirement_mpc_LR}
\end{equation}
which is a slightly modified version of $\mem{LR}$ in~\eqref{eq:mem_requirement_LR}.

\section{Experimental Evaluation}
\label{sec:experimental_results}
In this section, the memory requirement for storing the solution and the critical regions for the proposed method is compared to storing the full solution and critical regions. The \mpQP problems have been solved using \MPT (version ''3.1.2 (R2011a) 28.10.2015'') in \textsc{Matlab} (version {''8.4.0.150421 (R2014b)''}). Although the optimal active sets are computed in the \mpQP solver, it was not possible to access them in the solution returned by the solver. Hence, they have instead been retrieved by solving the corresponding \QP problem with the Chebychev center of each critical region as parameter.

\subsection{Defining the problems}
\label{subsec:experimental_problems}
 The comparison has been made for three different examples where explicit \MPC controllers are applied to stable \lti systems. For the first system, referred to as \emph{Problem~1}, the continuous system $1/(s+1)^{\nx}$
which is used in~\cite{kvasnica15:region_free} is studied, and the system has $\nx$ states and $\nuu = 1$ control inputs. The transfer function has been discretized using a unit sampling time, the weight matrices are $\QMPC = I$ and $\RMPC = I$, and the terminal cost $\PN$ is chosen as the discrete time \LQ cost. The states and control inputs are subject to the constraints
\begin{equation}
-10 \preceq x_t \preceq 10, \, t \!\in\! \intset{0}{N}, \; -1 \leq u_t \leq 1, \, t \!\in \!\intset{0}{N-1}.
\end{equation}
After re-writing this explicit \MPC problem into an equivalent \mpQP problem as in Section~\ref{sec:explicit_MPC}, the problem has $\np=\nx$ parameters and $\nz = N\nuu = N$ variables.

The second and third problems both use a system which is similar to the one used in, \eg,~\cite{wang10:fast_MPC,axehill15:_contr_mpc}. It consists of $n_M$ unit masses which are coupled with springs and dampers. The spring constant is chosen as $1$, the damping constant as~$0$, the weight matrices to $\QMPC = 100 I$ and $\RMPC = I$ and the terminal cost $\PN$ is chosen as the discrete time \LQ cost. The continuous system is discretized using the sampling time $0.5$ seconds. Two different cases have been studied, referred to as \emph{Problem~2} and \emph{Problem~3}. In Problem~2 the control input is a force acting between terra firma and the first mass, and in Problem~3 there is also an extra control input acting as a force applied between the first two masses. In both problems, the states and control inputs are subject to the constraints
\begin{equation}
-4\! \preceq x_t \preceq \!4, \, t\! \in \! \intset{0}{N}, \; -0.5 \preceq u_t \preceq 0.5, \, t \! \in \! \intset{0}{N-1}. 
\end{equation}
Each mass introduces $2$ states, and in Problem~2 $\nuu = 1$ and in Problem~3 $\nuu = 2$ by construction. Hence, by re-writing the \MPC problem into the equivalent \mpQP problem as in Section~\ref{sec:explicit_MPC}, the corresponding \mpQP problem has $\np = 2 n_M$ parameters and $\nz = N\nuu$ variables.

\vspace{-8pt}
\subsection{Experimental results}
\label{subsec:experimental_results}

The relative memory reduction has been computed for the three problems for different parameter dimensions and prediction horizons, and the results are summarized in tables~\ref{tab:ex1}-\ref{tab:masses_nu2}. The storage tree $\sTree{\Aset}{\rr}$ is chosen such that the the root node corresponds to the optimal active set $\as{\rr} = \{\}$, \ie, the unconstrained minimum, and whenever it is possible only one constraint is added to the children of each node. When the \mpQP problem is defined, all redundant constraints are removed, giving $\nc$ number of non-redundant constraints. Furthermore, $\R$ is the number of regions, $\tDepth$ is the maximum depth of $\sTree{\Aset}{\rr}$, and
\begin{equation}
\ratCR{LR}{F} \triangleq \frac{\memCR{LR}}{\memCR{F}}, \; \rat{LR}{F} \triangleq \frac{\mem{LR}}{\mem{F}}, \; \ratMPC{LR}{F} \triangleq \frac{\memMPC{LR}}{\memMPC{F}},
\end{equation}
are the relative reductions in the number of stored real numbers for only the critical regions, the full solution and the critical regions, and for storing the first $\nuu$ control inputs and the critical regions, respectively. Hence, for the case of explicit \MPC, $\ratMPC{LR}{F}$ determines the total relative memory reduction, whereas for a full \mpQP problem it is given by $\rat{LR}{F}$. No symmetry or other properties which are inherited from the explicit \MPC problems are exploited in the comparison. In some cases, for large problems with many parameters and variables the solution in a few regions are numerically bad. This might be a consequence of difficulties with finding the correct optimal active set given the critical region.

In Table~\ref{tab:ex1} the result for Problem~$1$ is seen, and it is clear that the relative memory reduction becomes increasingly beneficial with the parameter dimension. For $N=2$ and $N=3$, the memory is reduced by approximately an order of magnitude for the parameter dimensions $\np = 12$ and $\np=14$. The result for Problem~$2$ is presented in Table~\ref{tab:masses_nu1}, and for this problem the relative memory reduction is an order of magnitude for problems with $n_M \geq 6$. Table~\ref{tab:masses_nu2} contains the numerical results for Problem~$3$, and it can be seen that also for this problem the relative memory reduction is increased for larger parameter dimensions. Note that for all evaluated problems the memory required for the storage tree is lower than for storing the full solution.

\begin{table}
\caption{Experimental results for Problem 1.}
\centering
\label{tab:ex1}
\begin{tabular}{c||c|c|c||c|c|c}
$n/N$ & $\nc$ & $\R$ & $\tDepth$ &  $\ratCR{LR}{F}$ & $\rat{LR}{F}$ & $\ratMPC{LR}{F}$  \\
\hline
\hline
$2/2$ & 10 & 5 & 2 & 0.909 & 0.729 & 0.827 \\ 
$4/2$ & 20 & 11 & 2 & 0.446 & 0.403 & 0.431 \\ 
$6/2$ & 28 & 45 & 2 & 0.258 & 0.239 & 0.252 \\ 
$8/2$ & 44 & 153 & 2 & 0.176 & 0.167 & 0.173 \\ 
$10/2$ & 52 & 192 & 2 & 0.130 & 0.126 & 0.129 \\ 
$12/2$ & 66 & 255 & 2 & 0.107 & 0.105 & 0.107 \\ 
$14/2$ & 80 & 336 & 2 & 0.090 & 0.088 & 0.090 \\ 
\hline
$2/3$ & 12 & 5 & 2 & 0.909 & 0.694 & 0.827 \\ 
$4/3$ & 24 & 13 & 3 & 0.476 & 0.411 & 0.456 \\ 
$6/3$ & 36 & 89 & 3 & 0.258 & 0.234 & 0.252 \\ 
$8/3$ & 56 & 575 & 3 & 0.187 & 0.175 & 0.184 \\ 
$10/3$ & 66 & 1186 & 3 & 0.139 & 0.133 & 0.138 \\ 
$12/3$ & 86 & 1679 & 3 & 0.124 & 0.119 & 0.123 \\ 
$14/3$ & 102 & 2664 & 3 & 0.115 & 0.111 & 0.114 \\ 
\hline
$2/4$ & 14 & 5 & 2 & 0.909 & 0.667 & 0.827 \\ 
$4/4$ & 26 & 13 & 3 & 0.476 & 0.400 & 0.456 \\ 
$6/4$ & 42 & 129 & 4 & 0.252 & 0.226 & 0.246 \\ 
$8/4$ & 66 & 1222 & 4 & 0.207 & 0.189 & 0.203 \\ 
$10/4$ & 80 & 4300 & 4 & 0.161 & 0.151 & 0.159 \\ 
$12/4$ & 104 & 5408 & 4 & 0.181 & 0.172 & 0.179 \\ 
\hline
\end{tabular}
\end{table}

\begin{table}
\caption{Experimental results for Problem 2.}
\centering
\label{tab:masses_nu1}
\begin{tabular}{c||c|c|c||c|c|c}
$n_M/N$ & $\nc$ & $\R$ & $\tDepth$ &  $\ratCR{LR}{F}$ & $\rat{LR}{F}$ & $\ratMPC{LR}{F}$  \\
\hline
\hline
$2/2$ & 28 & 45 & 2 & 0.392 & 0.351 & 0.378 \\ 
$3/2$ & 40 & 161 & 2 & 0.220 & 0.206 & 0.217 \\ 
$4/2$ & 52 & 225 & 2 & 0.159 & 0.153 & 0.158 \\ 
$5/2$ & 64 & 229 & 2 & 0.131 & 0.127 & 0.130 \\ 
$6/2$ & 76 & 238 & 2 & 0.102 & 0.100 & 0.102 \\ 
$7/2$ & 88 & 239 & 2 & 0.087 & 0.086 & 0.087 \\ 
$8/2$ & 100 & 238 & 2 & 0.082 & 0.081 & 0.082 \\ 
\hline
$2/3$ & 38 & 127 & 3 & 0.393 & 0.341 & 0.379 \\ 
$3/3$ & 54 & 920 & 3 & 0.244 & 0.222 & 0.239 \\ 
$4/3$ & 70 & 1953 & 3 & 0.169 & 0.159 & 0.167 \\ 
$5/3$ & 86 & 2577 & 3 & 0.132 & 0.127 & 0.131 \\ 
$6/3$ & 102 & 2861 & 3 & 0.102 & 0.100 & 0.102 \\ 
$7/3$ & 118 & 3096 & 3 & 0.086 & 0.085 & 0.086 \\ 
$8/3$ & 134 & 3084 & 3 & 0.078 & 0.077 & 0.078 \\ 
\hline
$2/4$ & 48 & 282 & 4 & 0.406 & 0.336 & 0.388 \\ 
$3/4$ & 68 & 2593 & 4 & 0.275 & 0.242 & 0.268 \\ 
$4/4$ & 88 & 9479 & 4 & 0.203 & 0.187 & 0.200 \\ 
$5/4$ & 108 & 18707 & 4 & 0.148 & 0.140 & 0.146 \\ 
$6/4$ & 128 & 24629 & 4 & 0.111 & 0.108 & 0.111 \\ 
\hline
\end{tabular}
\end{table}

\begin{table}
\caption{Experimental results for Problem 3.}
\centering
\label{tab:masses_nu2}
\begin{tabular}{c||c|c|c||c|c|c}
$n_M/N$ & $\nc$ & $\R$ & $\tDepth$ &  $\ratCR{LR}{F}$ & $\rat{LR}{F}$ & $\ratMPC{LR}{F}$  \\
\hline
\hline
$2/2$ & 28 & 45 & 2 & 0.392 & 0.351 & 0.378 \\ 
$3/2$ & 40 & 161 & 2 & 0.220 & 0.206 & 0.217 \\ 
$4/2$ & 52 & 225 & 2 & 0.159 & 0.153 & 0.158 \\ 
$5/2$ & 64 & 229 & 2 & 0.131 & 0.127 & 0.130 \\ 
\hline
$2/3$ & 38 & 127 & 3 & 0.393 & 0.341 & 0.379 \\ 
$3/3$ & 54 & 920 & 3 & 0.244 & 0.222 & 0.239 \\ 
$4/3$ & 70 & 1953 & 3 & 0.169 & 0.159 & 0.167 \\ 
\hline
$2/4$ & 48 & 282 & 4 & 0.406 & 0.336 & 0.388 \\ 
$3/4$ & 68 & 2593 & 4 & 0.275 & 0.242 & 0.268 \\ 
\hline
\end{tabular}
\end{table}

\vspace{-10pt}
\section{Conclusions and Future Work}
\vspace{-4pt}
In this paper theory and algorithms for reducing the memory footprint when storing parametric solutions to \mpQP problems are introduced. This is performed by exploiting low rank structure in the parametric solutions. The structured changes in the parametric solution between neighboring critical regions is exploited in a similar way as low rank modifications is used as a tool to increase on-line performance in many popular \QP methods, but here it is applied to \mpQP problems to reduce the memory required to store the solutions. The proposed method stores the solution in a storage tree and can be implemented in already existing solvers for \mpQP problems, or be considered as a post-processing data compression step. For future work, an extension to other problem classes such as, \eg, multiparametric linear programming will be investigated. Furthermore, it will be studied which point location algorithm that benefits most by using the storage tree introduced in this paper, and also how to exploit low rank modifications of, \eg, Cholesky factorizations to improve the numerical properties.

\appendix
\section{Appendix}
\subsection{Linear algebra and definitions}
\label{app:lin_alg}
Consider the symmetric positive definite matrix $\begin{smallbmatrix}
\W & \w \\ \w^T & \wo
\end{smallbmatrix}$.

\vspace{1pt}
By using the matrix inversion lemma, the inverse of the block matrix is given by
\begin{equation}
\inv{\begin{smallbmatrix}
\W & \w \\ \w^T & \wo
\end{smallbmatrix}} \! = \begin{smallbmatrix}
\inv{\W} \! + \inv{\W} \w \inv{\C} \w^T \inv{\W} &  -\inv{\W}\w \inv{\C} \\ - \inv{\C} \w^T \inv{\W} & \inv{\C}
\end{smallbmatrix}, \label{eq:app:mtrx_inversion_lemma}
\end{equation} 
where $\C \triangleq \wo - \w^T \inv{\W} \w \in \posdefmats$.

The matrices in the \cftoc problem are defined by
\begin{subequations}
\label{eq:app:MPC_mtrx_def}
\begin{align}
&\timestack{A} \! \triangleq \! \begin{smallbmatrix}
I \\ A \\ \vdots \\ A^{\nx}
\end{smallbmatrix}, \, \timestack{B} \triangleq \begin{smallbmatrix}
0 \\
B \\
AB & B \\
\vdots \\
A^{\nx-1}B & A^{\nx-2}B & \cdots & B
\end{smallbmatrix}, \\
&\timestack{Q_x} \! \triangleq \textrm{diag}(\QMPC,...,\QMPC,\PN), \; \timestack{Q_u} \!\triangleq \! \textrm{diag}(\RMPC,...,\RMPC), \\
&\timestack{H_x} \! \triangleq \! \begin{smallbmatrix}
\HxMPC \\
 & \ddots \\
 & & \HxMPC
\end{smallbmatrix}\!, \, \timestack{H_u} \! \triangleq \!\begin{smallbmatrix}
\HuMPC \\
& \ddots \\
 & & \HuMPC
\end{smallbmatrix}\!, \timestack{h} \triangleq \begin{smallbmatrix}
\hMPC \\ \vdots \\ \hMPC
\end{smallbmatrix}.
\end{align}
\end{subequations}

\bibliography{axe_full,IEEEfull,ianFull}

\begin{thebibliography}{10}
\providecommand{\url}[1]{#1}
\csname url@samestyle\endcsname
\providecommand{\newblock}{\relax}
\providecommand{\bibinfo}[2]{#2}
\providecommand{\BIBentrySTDinterwordspacing}{\spaceskip=0pt\relax}
\providecommand{\BIBentryALTinterwordstretchfactor}{4}
\providecommand{\BIBentryALTinterwordspacing}{\spaceskip=\fontdimen2\font plus
\BIBentryALTinterwordstretchfactor\fontdimen3\font minus
  \fontdimen4\font\relax}
\providecommand{\BIBforeignlanguage}[2]{{%
\expandafter\ifx\csname l@#1\endcsname\relax
\typeout{** WARNING: IEEEtran.bst: No hyphenation pattern has been}%
\typeout{** loaded for the language `#1'. Using the pattern for}%
\typeout{** the default language instead.}%
\else
\language=\csname l@#1\endcsname
\fi
#2}}
\providecommand{\BIBdecl}{\relax}
\BIBdecl

\bibitem{bank82:nonlinear_par_opt}
B.~Bank, J.~Guddat, D.~Klatte, B.~Kummer, and K.~Tammer, \emph{Non-linear
  parametric optimization}, ser. Mathematische Lehrb{\"u}cher und Monographien:
  Mathematische Monographien.\hskip 1em plus 0.5em minus 0.4em\relax
  Birkh{\"a}user Verlag, 1982.

\bibitem{Bemporad02:explicit_mpc}
A.~Bemporad, A.~Morari, V.~Dua, and E.~Pistikopoulos, ``The explicit linear
  quadratic regulator for constrained systems,'' \emph{Automatica}, vol.~38,
  no.~1, pp. 3 -- 20, 2002.

\bibitem{maciejowski2002predictive}
J.~Maciejowski, \emph{Predictive control with constraints}.\hskip 1em plus
  0.5em minus 0.4em\relax Prentice Hall, 2002.

\bibitem{jonson83:thesis}
H.~Jonson, ``A {N}ewton method for solving non-linear optimal control problems
  with general constraints,'' Ph.D. dissertation, Link\"{o}pings Tekniska
  H\"{o}gskola, 1983.

\bibitem{rao98:_applic_inter_point_method_model_predic_contr}
C.~Rao, S.~Wright, and J.~Rawlings, ``Application of interior-point methods to
  model predictive control,'' \emph{Journal of Optimization Theory and
  Applications}, vol.~99, no.~3, pp. 723--757, Dec. 1998.

\bibitem{nielsen13low_rank_updates}
I.~Nielsen, D.~Ankelhed, and D.~Axehill, ``Low-rank modification of {R}iccati
  factorizations with applications to model predictive control,'' in
  \emph{Proceedings of the 52nd {IEEE} Conference on Decision and Control},
  Firenze, Italy, Dec. 2013, pp. 3684--3690.

\bibitem{kvasnica15:region_free}
M.~Kvasnica, B.~Tak{\'{a}}cs, J.~Holaza, and S.~D. Cairano, ``On region-free
  explicit model predictive control,'' in \emph{Proceedings of the 54th {IEEE}
  Conference on Decision and Control}, Osaka, Japan, Dec. 2015, pp. 3669--3674.

\bibitem{tondel03:binary_tree}
P.~T{\o}ndel, T.~Johansen, and A.~Bemporad, ``Evaluation of piecewise affine
  control via binary search tree,'' \emph{Automatica}, vol.~39, no.~5, pp. 945
  -- 950, 2003.

\bibitem{fuchs10:on_the_choice_linear}
A.~N. Fuchs, D.~Axehill, and M.~Morari, ``On the choice of the linear decision
  functions for point location in polytopic data sets - application to explicit
  mpc,'' in \emph{Proceedings of the 49th {IEEE} Conference on Decision and
  Control}, Dec 2010, pp. 5283--5288.

\bibitem{tondel03:mpQP_and_eMPC}
P.~T{\o}ndel, T.~A. Johansen, and A.~Bemporad, ``An algorithm for
  multi-parametric quadratic programming and explicit {MPC} solutions,''
  \emph{Automatica}, vol.~39, no.~3, pp. 489 -- 497, 2003.

\bibitem{gupta11:novel_mpQP}
A.~Gupta, S.~Bhartiya, and P.~Nataraj, ``A novel approach to multiparametric
  quadratic programming,'' \emph{Automatica}, vol.~47, no.~9, pp. 2112--2117,
  2011.

\bibitem{jones06:mpLCP}
C.~N. Jones and M.~Morrari, ``Multiparametric linear complementarity
  problems,'' in \emph{Proceedings of the 45th IEEE Conference on Decision and
  Control}, Dec 2006, pp. 5687--5692.

\bibitem{borrelli10:linear_mpc_laws}
F.~Borrelli, M.~Baoti\'{c}, J.~Pekar, and G.~Stewart, ``On the computation of
  linear model predictive control laws,'' \emph{Automatica}, vol.~46, no.~6,
  pp. 1035 -- 1041, 2010.

\bibitem{nocedal06:num_opt}
J.~Nocedal and S.~Wright, \emph{Numerical Optimization}.\hskip 1em plus 0.5em
  minus 0.4em\relax Springer-Verlag, 2006.

\bibitem{kvasnica12:reducing_memory}
M.~Kvasnica, J.~Hled\'{i}k, and M.~Fikar, ``Reducing the memory footprint of
  explicit mpc solutions by partial selection,'' in \emph{Proceedings of the
  51st {IEEE} Conference on Decision and Control}, Maui, Hawaii, USA, 2012.

\bibitem{herceg13:eMPC_graph}
M.~Herceg, S.~Mariethoz, and M.~Morari, ``Evaluation of piecewise affine
  control law via graph traversal,'' in \emph{Proceeding of the 2013 European
  Control Conference}, July 2013, pp. 3083--3088.

\bibitem{MPT3}
M.~Herceg, M.~Kvasnica, C.~Jones, and M.~Morari, ``{Multi-Parametric Toolbox
  3.0},'' in \emph{Proceedings~of the European Control Conference}, Z\"urich,
  Switzerland, July 17--19 2013, pp. 502--510,
  \url{http://control.ee.ethz.ch/~mpt}.

\bibitem{wang10:fast_MPC}
Y.~Wang and S.~Boyd, ``Fast model predictive control using online
  optimization,'' \emph{IEEE Transactions on Control Systems Technology},
  vol.~18, no.~2, pp. 267--278, March 2010.

\bibitem{axehill15:_contr_mpc}
D.~Axehill, ``Controlling the level of sparsity in {MPC},'' \emph{{Systems \&
  Control Letters}}, vol.~76, pp. 1--7, 2015.

\end{thebibliography}

\end{document}